\documentclass[12pt]{amsart}
 \usepackage{amsmath,amssymb}
\usepackage[centering,textwidth=6.5in,textheight=9in,marginpar=0.7in]{geometry}
  \usepackage{eucal}
   \usepackage{caption}
   \usepackage{mathtools}
   \usepackage[new]{old-arrows}
   \usepackage{stmaryrd}
   \usepackage{pigpen}
\usepackage{mathrsfs, upgreek}
\usepackage{leftidx}
\usepackage{tikz,tikz-3dplot}
\usepackage{tikz-cd}
\usepackage{enumerate}
\usepackage{array}
\usetikzlibrary{arrows,matrix,positioning}
\usetikzlibrary{patterns,decorations.pathreplacing, decorations.markings}
\usetikzlibrary{math,angles,quotes,intersections,calc,through}

 \usepackage[all]{xy}

\usepackage[plainpages=false,colorlinks,hyperindex,pdfpagemode=None,bookmarksopen,linkcolor=blue,citecolor=blue,urlcolor=blue]{hyperref}

\newtheorem{cor}[subsection]{Corollary}
\newtheorem{thm}[subsection]{Theorem}

\newtheorem{def/lemma}[subsection]{Definition/Lemma}
\newtheorem{conjecture}[subsection]{Conjecture}
\newtheorem{prop}[subsection]{Proposition}
\theoremstyle{definition}

\newtheorem{remark}[subsection]{Remark}
\newtheorem*{remark*}{Remark}

\newtheorem*{theorem*}{Theorem}
\newtheorem*{defn*}{Definition}
\newtheorem*{lemma*}{Lemma}
\newtheorem*{corollary*}{Corollary}
\newtheorem{question*}{Question}
\newtheorem*{conjecture*}{Conjecture}
 \newtheorem*{prop*}{Proposition}
 \newtheorem*{example*}{Example}

\newcommand{\cB}{\mathcal{B}}
\newcommand{\cC}{\mathcal{C}}

\newcommand{\cM}{\mathcal{M}}

\newcommand{\cO}{\mathcal{O}}

\newcommand{\cS}{\mathcal{S}}

\newcommand{\cZ}{\mathcal{Z}}

\newcommand{\bA}{\mathbb{A}}
\newcommand{\bC}{\mathbb{C}}

\newcommand{\bF}{\mathbb{F}}
\newcommand{\bG}{\mathbb{G}}

\newcommand{\bP}{\mathbb{P}}
\newcommand{\bQ}{\mathbb{Q}}
\newcommand{\bR}{\mathbb{R}}

\newcommand{\bZ}{\mathbb{Z}}

\newcommand{\ft}{\mathfrak{t}}
\newcommand{\fC}{\mathfrak{C}}

\newcommand{\Spec}{\mathrm{Spec}}

\numberwithin{equation}{subsection}

\newcommand{\reg}{\mathrm{reg}}

\newcommand{\Ad}{\mathrm{Ad}}
\newcommand{\fg}{\mathfrak{g}}

\newcommand{\fb}{\mathfrak{b}}
\newcommand{\fn}{\mathfrak{n}}

\newcommand{\fl}{\mathfrak{l}}
\newcommand{\fsl}{\mathfrak{sl}}

\newcommand{\der}{\mathrm{der}}
\newcommand{\ad}{\mathrm{ad}}

\newcommand{\lng}{\langle}
\newcommand{\rng}{\rangle}

\newcommand{\cpt}{\mathsf{cpt}}

\newcommand{\dagg}{\dagger}

\newcommand{\sfE}{\mathsf{E}}

\newcommand{\rank}{\textup{rank }}

\newcommand{\po}{\ar@{}[dr]|{\text{\pigpenfont R}}}
\newcommand{\pb}{\ar@{}[dr]|{\text{\pigpenfont J}}}

\begin{document}

\title{Cohomology of the universal centralizers I: the adjoint group case}
\author{Xin Jin}
\address{Math Department, Boston College, Chestnut Hill, MA 02467, United States of America.}
\email{xin.jin@bc.edu}

\begin{abstract}
We compute the rational cohomology of the universal centralizer $J_G$ (also known as the Toda system or BFM space) for a complex (connected) semisimple group $G$ of adjoint form. While $J_G$ exhibits interesting and increasingly complex topology as the rank of $G$ rises, its rational cohomology is surprisingly simple--it coincides with that of a point. In a subsequent work \cite{Jin2}, we will extend this analysis to the case of $J_G$ for general semisimple $G$. In particular, we will show that its rational cohomology has pure Hodge structure.
\end{abstract}

\maketitle

\tableofcontents

\section{Introduction}

The  universal centralizer  $J_G$ (cf. \cite{Lusztig, Kostant, BFM, Teleman, Ginzburg}\footnote{In \cite{Lusztig}, the group-group version of universal centralizer was first introduced, which is different from the Lie algebra-group version $J_G$ considered in the current paper.}) is a smooth affine completely integrable system associated to a (connected) complex semisimple (or reductive) algebraic group $G$. In particular, it has a natural holomorphic symplectic structure. It has appeared in various contexts of geometric representation theory, differential geometry and mathematical physics. 
For example, it is identified with the moduli space of solutions to the Nahm equations, as in work of Atiyah-Hitchin \cite{AH}, Donaldson \cite{Donaldson}, Bielawski \cite{Bielawski}, etc.. Moreover, it is the Coulomb branch with matter $0$, mathematically defined by Braverman-Finkelberg-Nakajima \cite{BFN}. 
It was used in Ngo's proof of the Fundamental Lemma \cite{Ngo}, and it has recently received much more attention due to the study of bi-Whittaker $D$-modules, which has important applications in geometric representation theory and geometric Langlands program (cf. \cite{BZG, Lonergan, Ginzburg, Gannon}). There is also a Betti version of the result, where the category of  bi-Whittaker $D$-modules is replaced by the wrapped Fukaya category of $J_G$ (equivalently, microlocal sheaf category). This is stated as a conjecture in Ben-Zvi--Gunningham \cite[Remark 2.7]{BZG} and is presented by the author  in \cite{Jin} as a homological mirror symmetry result.

There are some natural analogies between $J_G$ (more generally, symplectically resolved Coulomb branches) and (smooth) Hitchin integrable systems (cf. \cite{BFN, Maxence}). On the other hand, $J_G$ is an open subset of the affine Toda system $\cM_G$ (cf. \cite{Etingof} and references therein for the definition of the affine Toda system), which is naturally identified with a smooth moduli space of Higgs bundles on $\bP^1$ with certain automorphic data (in particular irregular singularities of the Higgs fields) at $0, \infty$ (cf. the upcoming work \cite{JY}). It is then natural to expect that $J_G$ has a pure Hodge structure, as all (smooth) Hitchin systems do.  

On the other hand, there are essential differences between $J_G$ and usual (smooth) Hitchin integrable systems. First, the completely integrable system associated with $J_G$ is non-proper, and has a generic fiber isomorphic to a maximal torus in $G$. Second, the partial symplectic compactification \( \mathcal{M}_G \) (more precisely, its neutral component), viewed as a wild Higgs moduli space with a proper Hitchin map, does not possess a \( \mathbb{C}^\times \)-action that contracts everything to a central Hitchin fiber. Both properness (over the base) and the $\bC^\times$-action are essential for the argument that the usual (smooth) Higgs moduli spaces exhibit a pure Hodge structure (cf. \cite[\S 5]{Felisetti}). In a different direction, there is a partial log-compactification of $J_G$ (only for adjoint groups) given by Balibanu \cite{Balibanu} that possesses both properness and \( \mathbb{C}^\times \)-action, but it is not symplectic, in particular not a Hitchin system. 

The purpose of this note and its sequel \cite{Jin2} is to show that $J_G$ has pure Hodge structure, and give an explicit answer for its rational cohomology. 

 Let $n=\rank G$ and let $\Pi$ be a complete set of simple roots. In \cite{Teleman, Jin},  a Bruhat stratification of $J_G$ is given (see Proposition \ref{prop: Bruhat} (i)), in which all strata are locally closed subvarieties that have connected components isomorphic to a product of $\bA^n$ and $(\bG_m)^{k}$. Hence one can explicitly calculate the $\bF_q$-points of $J_G$ (for $q$ sufficiently large) and the $\sfE$-polynomial:
\begin{align*}
&|J_G(\bF_q)|=q^{n}\sum_{S\subset \Pi} |\pi_0(Z(L_S))|(q-1)^{n-|S|}.
\end{align*}
and $\sfE_{J_G}(u,v)$ is given by plugging into $q$ by $uv$.
 If $J_G$ has pure Hodge structure, then the Poincar\'e polynomial can be obtained by $t^{4n}\sfE_{J_G}(-t^{-1},-t^{-1})$ (cf. \cite[\S 4.2]{Hausel}). 
 In the case when $G$ is of adjoint form, $|\pi_0(Z(L_S))|=1$ for all $S$, hence  $|J_G(\bF_q)|=q^{2n}$.
 This motivates our main result. 

\begin{thm}\label{prop: RHJ_G_ad}
Assume $G$ is of adjoint form. Then the rational cohomology of $J_{G}$ is trivial.  
\end{thm}

We remark that the topology of $J_G$ becomes increasingly complex as $\rank G$ grows. In particular, the integral cohomology has complicated torsions.  
Our strategy for the proof of the theorem is by induction on the rank of $G$, and the key parabolic induction pattern of the geometry of $J_G$ reviewed in \S\ref{subsubsec: Induction pattern}--\S\ref{subsubsec: handle}
. This strategy also extends to semisimple groups $G$ with a nontrivial center, where the rational cohomology is nontrivial, but substantial more work is needed to take care of the non-triviality of $\pi_0(Z(L_S)), S\subset\Pi$, and its impact on inductions. This will be included in a forthcoming paper \cite{Jin2}.

We make the following conjecture that is natural from the analogies between Coulomb branches and Hitchin systems. 
\begin{conjecture}
The rational cohomology of every (resolved) Coulomb branch associated with connected semisimple groups, in the sense of Braverman-Finkelberg-Nakajima, has pure Hodge structure. 
\end{conjecture}

\subsubsection*{{\bf Acknowledgement}}
I am grateful to Maxence Mayrand, Junliang Shen and Zhiwei Yun for useful discussions related to $\sfE$-polynomials and mixed Hodge structures. I also thank George Lusztig for his interest in this work.

\section{Preliminaries}

\subsection{Set-up}\label{subsec: set-up}
For any complex algebraic group $H$, let $H_0$ be the identity component. When $H$ is reductive, let $H_{\cpt}$ be the maximal compact subgroup of $H$. 
Let $G$ be a connected complex semisimple group with Lie algebra $\fg$. Let $G_\ad$ (resp. $G_{sc}$) be the adjoint form (resp. simply connected form) of $G$. Fix a maximal torus and a Borel subgroup $T\subset B \subset G$ with Lie algebras $\ft\subset \fb\subset \fg$, and let  $\Pi$ be the associated set of simple roots. Let $W=N_G(T)/T$ be the Weyl group. For any $\alpha\in \Pi$, let $\lambda_{\alpha^\vee}$ be the fundamental weight dual to $\alpha^\vee$. Let $N\subset B$ be the unipotent radical of $B$ and let $\fn$ be the Lie algebra of $N$ (similarly, we have the opposite Borel $B^-$ and $N^-\subset B^-$ with Lie algebras $\fn^-\subset \fb^-$). 
For any closed subgroup $H\subset G$ that contains $Z(G)$ (the center of $G$), let $H_\ad$ (resp. $H_{sc}$) be the corresponding quotient in $G_{\ad}$ (resp. preimage in $G_{sc}$). 

For any $S\subset \Pi$, let $L_S$ be the standard Levi subgroup associated to $S$. Let $W_S\subset W$ be the Weyl group of $L_S$, which is the subgroup generated by simple reflections from $S$. Let $w_0\in W$ (resp. $w_S\in W_S$) be the longest element. Let $L_S^\der=[L_S, L_S]$ be the derived subgroup of $L_S$. Let $L_{S;\ad}=L_S/Z(L_S)$. Let $Z(L_S^\der)_0=Z(L_S^\der)\cap Z(L_S)_0$. 
For any $w\in W$, let $\dot{w}$ be a lifting of $w$ in $N_G(T)$. 

For any finite nonempty set $Q$, let $\fC^{Q_\dagg}$ be the standard simplex with vertices indexed by elements from $Q$ (so the dimension of the simplex is $|Q|-1$). 
For any $S\subsetneq \Pi$, set $\fC_{S}:=(\fC^{(\Pi-S)_\dagg})^\circ$ be the interior of the face in $\fC^{\Pi_\dagg}$ whose set of vertices is  $\Pi-S$. For example, if $S=\emptyset$, then $\fC_{\emptyset}$ is the interior of $\fC^{\Pi_\dag}$. The collection $\{\fC_S\}_{S\subsetneq \Pi}$ gives a standard stratification of $\fC^{\Pi_\dag}$. 

For any complex torus $T'$, let $X_\bullet(T')$ (resp. $X^\bullet(T')$) be the cocharacter lattice (resp. character lattice). For the maximal torus $T$ as above, let $X^+(T)$ be the semi-group of dominant characters with respect to $\Pi$.  

Let $\fg^{\reg}$ be the open subset of regular elements in $\fg$. Let $\{e,f, h\}$ be a fixed principal $\fsl_2$-triple with $e\in \bigoplus_{\alpha_i\in \Pi} \fg_{\alpha_i}$. Then $e$ has a nonzero component in every $\fg_{\alpha_i}$;  $f\in \bigoplus_{\alpha_i\in \Pi} \fg_{-\alpha_i}$ with a nonzero component in every $\fg_{-\alpha_i}$;  $h=2\delta^\vee\in \ft$, where $\delta^\vee$ is the sum of all fundamental coweights. 

Recall the Kostant slice $\cS=f+\ker \ad_e\subset \fg^\reg$, and the $N$-equivariant isomorphism $N\times \cS\overset{\sim}{\to}f+\fb$ taking $(u, \xi)\in N\times \cS$ to $\Ad_u\xi$. The Kostant slice has the important property that the composition
\begin{align*}
\cS\hookrightarrow \fg^\reg\to  \fg^\reg\sslash G\cong \fg\sslash G\cong \ft\sslash W
\end{align*}
is an isomorphism. In other words, $\cS$ is parametrizing \emph{regular} adjoint orbits in $\fg$.

\subsection{Review of the geometry of $J_G$}

In this subsection, we review some important geometric properties of $J_G$, especially the parabolic induction pattern and handle attachment structure. For more details, see \cite[Section 2]{Jin}. 

\subsubsection{Two equivalent definitions of $J_G$}\label{subsubsec: def of J_G}
Recall the two equivalent definitions of $J_G$:
\begin{itemize}
\item[(i)] $J_G:=\{(g,\xi): \Ad_g\xi=\xi\}\subset G\times \cS$;\\
\item[(ii)] Let $\mu: T^*G\to \fn^*\times \fn^*\cong \fn^-\times\fn^-$ be the moment map of the Hamiltonian $N\times N$-action on $T^*G$ induced from the left and right $N$-action on $G$. Then $(f,f)\in  \fn^-\times\fn^-$ is a regular Lie algebra character of $\fn\times \fn$, and $N\times N$-acts freely on $\mu^{-1}(f,f)$. Then 
\begin{align*}
J_G:=\mu^{-1}(f,f)/N\times N\cong\{(g,\xi): \xi\in f+\fb, \Ad_g\xi\in f+\fb\}/N\times N,
\end{align*}
which is called the \emph{bi-Whittaker reduction}. 
\end{itemize}

Definition (ii) endows $J_G$ with a canonical holomorphic symplectic form. 
Let $\chi: J_G\to\ft\sslash W, (g,\xi)\mapsto [\xi]$ be the characteristic map. This is a completely integrable system, in particular a commutative group scheme over the base. Over any closed point in the open locus $\ft^{\reg}\sslash W$ parametrizing regular semsimple conjugacy classes, the fiber is isomorphic to a maximal torus in $G$. 
Let $[0]$ be the image of $0\in \ft$ in $\ft\sslash W$. Then $\chi^{-1}([0])\cong C_G(f)\cong Z(G)\times C_G(f)_0$, where $C_G(f)_0$ is an abelian unipotent subgroup isomorphic to $\bG_a^{\rank G}$. 

Using Definition (i), we also have the Kostant sections $\Sigma_z:=\{(g=z, \xi):\xi\in \cS\}$ for $z\in Z(G)$. 

\subsubsection{The canonical $\bC^\times$-action on $J_G$}

There is a canonical $\bC^\times$-action on $J_G$ that scales the holomorphic symplectic form of $J_G$ by weight $2$.
Using definition (i) or (ii) in \S\ref{subsubsec: def of J_G} and the principal $\fsl_2$-triple in \S\ref{subsec: set-up}, it is defined as follows: 
\begin{align*}
s\cdot (g,\xi)=(\Ad_{s^h}g, s^2\Ad_{s^h}(\xi)), s\in \bC^\times.
\end{align*}
Its fixed points are $\{(g=z, \xi=f): z\in Z(G)\}$.

\subsubsection{The map $b_G$ and its tropicalization $|b_G|$}\label{subsubsec: b_G}
Definition (ii) of $J_G$ above induces a natural algebraic map 
\begin{align*}
b_G: J_G\to \Spec \cO(G/N)^N\cong \Spec\ \bC[X^+(T)],
\end{align*}
which is equivariant with respect to the canonical $\bC^\times$-action on $J_G$ and the induced one on $\Spec\ \bC[X^+(T)]$ (with $s\cdot x^{\lambda}=s^{-2\lng\lambda,h\rng}x^\lambda$, for $s\in \bC^\times$ and $\lambda\in X^+(T)$). 
Using that $X^+(T_{sc})=\sum_{\alpha\in \Pi}\bZ_{\geq 0}\lambda_{\alpha^\vee}$, the semi-group freely generated by the fundamental weights, there is a canonical isomorphism $\Spec\ \bC[X^+(T_{sc})]\cong \bA^{\Pi}$. 
Then $\Spec\ \bC[X^+(T)]= (\Spec\ \bC[X^+(T_{sc})])\sslash Z(G)\cong \bA^{\Pi}\sslash Z(G)$. Composing $b_G$ with the projection to the norm of the standard affine coordinates of $\bA^{\Pi}$ (which clearly descend to $\bA^{\Pi}\sslash Z(G)$), we get the tropicalization of $b_G$: 
\begin{align*}
|b_G|: J_G\to \bR_{\geq 0}^\Pi.
\end{align*}
The target  $\bR_{\geq 0}^\Pi$ is naturally stratified by $\bR_{>0}^{\Pi-S}\times \{0\in \bR_{\geq 0}^{S}\}$, for $S\subset\Pi$. Let $U_S:=\bigcup_{S^\dag\subset S}\bR_{>0}^{\Pi-{S^\dag}}\times \{0\in \bR_{\geq 0}^{S^\dag}\}$ be the open subset $\bR_{\geq 0}^\Pi$ consisting of strata indexed by $S^\dag\subset S$. 

\subsubsection{Parabolic induction pattern}\label{subsubsec: Induction pattern}

The following proposition combines  \cite[Proposition 2.3 and 2.6]{Jin} (see also \cite{Teleman}).
\begin{prop}\label{prop: Bruhat}
\begin{itemize}
\item[(i)] Using the bi-Whittaker reduction realization of $J_G$, we have $(g,\xi)\in |b_G|^{-1} (\bR_{>0}^{\Pi-S}\times \{0\in \bR_{\geq 0}^{S}\})$ if and only if $g\in B\dot{w}_0\dot{w}_S B$. This gives a Bruhat decomposition of $J_G=\bigsqcup_{S\subset\Pi}\cB_{w_0w_S}$, where 
\begin{align*}
\cB_{w_0w_S}:=|b_G|^{-1} (\bR_{>0}^{\Pi-S}\times \{0\in \bR_{\geq 0}^{S}\})\cong \Sigma_{I;S}\times T^*Z(L_S)\cong \fl_S^\der\sslash L_S^\der\times T^*Z(L_S),
\end{align*}
where $\Sigma_{I;S}$ is the identity Kostant section of $J_{L_{S}^\der}$. 
In particular, $|b_G|^{-1}(0)=\bigsqcup_{z\in Z(G)}\Sigma_z$.  

\item[(ii)] By choosing appropriate liftings $\dot{w}_S, S\subset \Pi$, there is a natural isomorphism of holomorphic symplectic varieties 
\begin{align*}
|b_G|^{-1}(U_S)\cong J_{L_S}=J_{L_S^\der}\overset{Z(L_S^\der)}{\times}T^*Z(L_S). 
\end{align*}

\item[(iii)]
For $S_1\subset S_2$, let $L_{S_2}^{S_1}=L_{S_1}\cap L_{S_2}^\der$. There is a compatible system of open inclusions $J_{L_{S_1}^{S_2}}\hookrightarrow J_{L_{S_2}^\der}$ for all pairs of $S_1\subset S_2$, so that the following diagram commutes
\begin{equation*}
\begin{tikzcd}
{|b_G|^{-1}(U_{S_1})}\ar[d, equal, "\wr{}"] \ar[r, hook]& {|b_G|^{-1}(U_{S_2})}\ar[d, equal, "\wr{}"]\\
{J_{L_{S_1}^{S_2}}\overset{Z(L_{S_2}^\der)}{\times}T^*Z(L_{S_2})}\ar[r, hook]&{J_{L_{S_2}^\der}\overset{Z(L_{S_2}^\der)}{\times}T^*Z(L_{S_2})}. 
\end{tikzcd}
\end{equation*}

\end{itemize}
\end{prop}

\begin{remark}
Another application of Proposition \ref{prop: Bruhat} (ii) and (iii) is that it gives an alternative description of the partial log-compactification given in \cite{Balibanu}\footnote{This discussion is independent with the rest of the paper, so the reader can safely skip it.}. Namely, for $G$ of adjoint form, 
\begin{align}\label{eq: J_G, log}
\overline{J}^{\log{}}_G=\bigcup_{S\subset \Pi} J_{L_S^\der}\overset{Z(L_S^\der)}{\times}T_D^*\overline{Z(L_S)}^0
\end{align}
where (1) $\overline{T}^0=\Spec\ \bC[\bZ_{\leq 0}^{\Pi}]$ is the partial compactification of $T$ in the big open cell $X_0$ of the wonderful compactification of $G$ (as in \cite[\S 2.2]{EJ}), and $\overline{Z(L_S)}^0\subset \overline{T}^0$ is the closed subvariety defined by $\alpha=1$ for $\alpha\in S$, which is a partial compactification of $Z(L_S)$; (2) $T_D^*\overline{Z(L_S)}^0$ is the log cotangent bundle associated with the normal crossing divisor $D=\overline{Z(L_S)}^0-Z(L_S)$. The gluing of the open affine pieces on the right-hand-side of \eqref{eq: J_G, log} is through the obvious open embeddings in the following correspondence, for any pair $S_1\subset S_2$: 
\begin{equation*}
\begin{tikzcd}
J_{L_{S_1}^\der}\overset{Z(L_{S_1}^\der)}{\times}T_D^*\overline{Z(L_{S_1})}^0&\ar[l, hook']J_{L_{S_1}^{S_2}}\overset{Z(L_{S_2}^\der)}{\times}T^*_D\overline{Z(L_{S_2})}^0\ar[r, hook]&J_{L_{S_2}^\der}\overset{Z(L_{S_2}^\der)}{\times}T^*_D\overline{Z(L_{S_2})}^0. 
\end{tikzcd}
\end{equation*}
The identification of $\overline{J}^{\log{}}_G$ with the log-compactification of Balibanu can be seen as follows. First, there is a well defined open embedding $\overline{J}^{\log{}}_G$ into the Whittaker reduction of $T^*_D\overline{G}$ in \cite[\S 3.2]{Balibanu}, using an extension of the formula \cite[(2.2.8)]{Jin}\footnote{One should first compose the cited formula with the embedding $G\times \fg\hookrightarrow G\times \fg\times \fg, (g, \xi)\mapsto (g, \xi, \Ad_g\xi)$, and then do the unique (well defined) extension  $T^*_D\overline{Z(L_S)}^0\times^{Z(L_S^\der)} \mu_{N_S\times N_S}^{-1}(f_S, f_S)\to\overline{G}\times\fg\times\fg$. } for each open piece on the right-hand-side of \eqref{eq: J_G, log}. Second, using the affine paving $X_J, J\subset \Pi$ of the log-compactification coming from the $\bC^\times$-action \cite[Proposition 4.11]{Balibanu}, it is easy to see that $X_J=\Sigma_{I;S}\times T^*_D\overline{Z(L_S)}^0$ for $S=\Pi-J$. Therefore, the embedding is an isomorphism. 

Note that the above perspective will greatly simplify the proof of \cite[Proposition 3.6]{Jin}, which will be included in a newer version of that paper soon. 
\end{remark}

\subsubsection{The complement of the Kostant sections}\label{subsubsec: complement Kostant}
Consider the map induced by $|b_G|$: 
\begin{align*}
\pi_{\cC_G}: \cC_G:=(J_G-\bigcup_{z\in Z(G)}\Sigma_z)/\bR_+\longrightarrow (\bR_{\geq 0}^\Pi-\{0\})/\bR_+\cong \fC^{\Pi_\dag}.
\end{align*}
where the identification $(\bR_{\geq 0}^\Pi-\{0\})/\bR_+\cong \fC^{\Pi_\dag}$ identifies $(\bR_{>0}^{\Pi-S}\times \{0\in \bR_{\geq 0}^{S}\})/\bR_+$ with $\fC_S$, for $S\subsetneq \Pi$. Choose any proper strictly positive homogeneous function $r$ on $\bR_{\geq 0}^\Pi-\{0\}$, then we can identify $\cC_G$ with $|b_G|^{-1}(r^{-1}(1))$. In particular, $\cC_G$ is a smooth (real contact) manifold. 

For any $S\subsetneq \Pi$, let $U_{S;\fC}=U_S/\bR_+$. Then for any $S_1\subset S_2\subsetneq \Pi$, we have the natural commutative diagram
\begin{equation}\label{eq: diagram pi_cC, U}
\begin{tikzcd}
{\pi_{\cC_G}^{-1}(U_{S_1;\fC})}\ar[r, hook]& {\pi_{\cC_G}^{-1}(U_{S_2;\fC})}\\
{J_{L_{S_1}^{\der}}\overset{Z(L_{S_1}^\der)_0}{\times}T^*Z(L_{S_1})_{0}}\ar[r, hook]\ar[u, "\overset{h.e.}{\sim}"] &{J_{L_{S_2}^\der}\overset{Z(L_{S_2}^\der)_0}{\times}T^*Z(L_{S_2})_{0}}\ar[u, "\overset{h.e.}{\sim}"],
\end{tikzcd}
\end{equation}
where $h.e.$ stands for homotopy equivalence.

\subsubsection{Handle attachment}\label{subsubsec: handle}

Let $\chi^{-1}([0])_z$ be the component of the central fiber of $\chi$ that contains $(z, f)$. Let $\chi^{-1}([0])_z^c:=\chi^{-1}([0])_z\cap |b_G|^{-1}(r^{-1}([0,1]))$. 
Then $J_{G}$ is  obtained as a topological space by attaching the (real) $2n$-dimensional cell $\chi^{-1}([0])_z^{c}$, for each $z\in Z(G)$, to $J_G-\bigcup_{z\in Z(G)}\Sigma_z$. 

Let $F_{h'}$ be a generic cotangent fiber in $|b_G|^{-1}(U_\emptyset)\cong T^*T$. 
Then $\chi^{-1}([0])_z$ and $F_{h'}$ are complex (Lagrangian) subvarieties in $J_G$, and $\chi^{-1}([0])_z\cap F_{h'}$ transversely in $\frac{|W|}{|Z(G)|}$-many points (cf. \cite[Proof of Proposition 5.2 in \S 6.3]{Jin} for $G$ of adjoint form; the general case follows easily from it). Note that $\chi^{-1}([0])_z$ is invariant under the canonical $\bC^\times$-action but $F_{h'}\subset |b_G|^{-1}(r^{-1}(\epsilon))$, for some $\epsilon\in \bR_+$, is not. 

\section{Proof of main result and some direct consequences}

In this section, we give the proof of the main result Theorem \ref{prop: RHJ_G_ad}. We will also deduce some direct consequences.

\begin{proof}[Proof of Theorem \ref{prop: RHJ_G_ad}]
We prove by induction on the rank of $G=G_{\ad}$.  For the case of rank $0$, there is nothing to prove. 
The rank $1$ case is also not hard to obtain:  $J_{PGL_2}\overset{h.e.}{\simeq} \bR\bP^2$, hence the statement holds\footnote{See \cite[Figure 2]{Jin} for a Lagrangian skeleton of $J_{SL_2}$ as a Weinstein sector; the quotient of the skeleton by the obvious free $\bZ/2\bZ$-symmetry (that identifies one cap to the other) gives the Lagrangian skeleton of $J_{PGL_2}$.}. Now assume $\rank G=n\geq 1$, and let $\Sigma_I$ be the Kostant section. Note that in this case every $Z(L_S)$ is connected.

Recall the notations from \S\ref{subsubsec: complement Kostant}. We will first use Mayer-Vietoris to calculate $H_{*}(\cC_G,\bQ)$ in the language of constructible (co)sheaves\footnote{All (co)sheaves are by default objects in the dg-category of (co)sheaves. All functors between sheaf categories are derived.}. 
For any locally compact Hausdorff space $X$, let $\omega_{X;\bQ}$ be the dualizing sheaf on $X$ over $\bQ$. Then $(\pi_{\cC_G})_!\omega_{\cC_G;\bQ}$ gives a constructible cosheaf (by taking $\Gamma_c$ on open subsets) on the simplex $\fC^{\Pi_\dagg}$ stratified by the faces $\fC_S$ indexed by $S\subsetneq \Pi$. 

 Then
\begin{align}\label{eq: Gamma_cUSad}
\Gamma_c(U_{S;\fC}, (\pi_{\cC_G})_!\omega_{\cC_G;\bQ})&\cong H_*\left(J_{L_S^\der}\overset{\cZ(L_S^\der)}{\times} T^*\cZ(L_S), \bQ\right)\\
\nonumber&\cong \left(H_*(J_{L_S^\der},\bQ)\otimes H_*(\cZ(L_S),\bQ)\right)^{\cZ(L_S^\der)}. 
\end{align}

Since $\cZ(L_S)\cong (\bC^\times)^{n-|S|}$, and from induction $H_*(J_{L_{S,\ad}}, \bQ)\cong \bQ$, 
\begin{align*}
H_{-*}(J_{L_S^\der}\overset{\cZ(L_S^\der)}{\times} T^*\cZ(L_S), \bQ)\cong H_{-*}(\cZ(L_S),\bQ)\cong  \Lambda^{*} (X_\bullet(\cZ(L_S))\otimes_\bZ \bQ[1]).
\end{align*}

In view of diagram \eqref{eq: diagram pi_cC, U}, $H_*\left(\cC_G, \bQ)\cong \Gamma_c(\fC^{\Pi_\dag}, (\pi_{\cC_G})_!\omega_{\cC_G;\bQ}\right)$ can be computed by the (homotopy) colimit of the diagram 
\begin{align}\label{eq: colimit, S}
(\{S\subsetneq \Pi\}, \subset)&\longrightarrow \text{Vect}_\bQ:=\text{the dg-category of }\bQ\text{-modules}\\
\nonumber S&\mapsto  H_{-*}(\cZ(L_S),\bQ)\cong \Lambda^{*} (X_\bullet(\cZ(L_S))\otimes \bQ[1]),
\end{align}
where the morphism $H_*(\cZ(L_{S}), \bQ)\rightarrow H_*(\cZ(L_{S'}),\bQ)$ for $S\subset S'$ is induced from the orthogonal projection $X_\bullet(\cZ(L_S))\otimes \bQ\rightarrow X_\bullet(\cZ(L_{S'}))\otimes \bQ$ with respect to the Killing form. 

On the other hand, using the set-up from \S\ref{subsubsec: b_G}, the colimit of \eqref{eq: colimit, S} is also calculating 
\begin{align*}
H_*(\bA^{\Pi}-\{0\},\bQ)\cong H_*(S^{2n-1},\bQ).
\end{align*}
Hence we get 
\begin{align*}
H_*(\cC_G, \bQ)\cong H_*(S^{2n-1},\bQ). 
\end{align*}

Lastly, using the handle attachment feature reviewed in \S\ref{subsubsec: handle},  it suffices to show that $\partial\chi^{-1}([0])^c$ is nontrivial in $H_{2n-1}(J_G-\Sigma_I, \bQ)\cong H_{2n-1}(\cC_G,\bQ)\cong H_{2n-1}(S^{2n-1},\bQ)\cong \bQ$.
In the following, we identify $\cC_G\cong |b_G|^{-1}(r^{-1}(1))$. 
 Let $F_{h'}$ be a generic cotangent fiber in $|b_G|^{-1}(U_\emptyset)\cong T^*T$.

Since $\chi^{-1}([0])\cap F_{h'}$ transversely in $|W|$-many points (cf. \S\ref{subsubsec: handle}), we can choose compatible orientation on $\partial\chi^{-1}([0])^c=\chi^{-1}([0])\cap \cC_G$ and co-orientation on $\bR_+\cdot F_{h'}$, so that 
the corresponding $(2n-1)$-cycle $[\chi^{-1}([0])\cap \cC_G]$ and $(2n-1)$-cocycle $[\bR_+\cdot F_{h'}]$ (equivalently, $(2n+1)$-Borel-Moore cycle) in $J_G-\Sigma_I$ satisfy 
\begin{align*}
[\chi^{-1}([0])\cap \cC_G]\cap [\bR_+\cdot F_{h'}]=|W|. 
\end{align*}
This shows that $H_{2n-1}(J_G-\Sigma_I,\bQ)\ni [\partial \chi^{-1}([0])^c]\neq 0$, and we have $H_*(J_{G}, \bQ)\cong \bQ$ as desired. 
\end{proof}

\begin{cor}\label{cor: pi0ZLStrivial,H_*}
Given any complex semisimple $G$ of rank $n$, assume that for any $S\subsetneq \Pi$, $\pi_0(\cZ(L_S))=1$, then 
\begin{align*}
H_{-*}(J_G,\bQ)\cong \bQ^{(|\cZ(G)|-1)}[2n]\oplus \bQ.
\end{align*}
Moreover, $H^*(J_G,\bQ)\cong H_{4n-*}^{BM}(J_G,\bQ)$ has a basis represented by the algebraic Borel-Moore cycles $[\Sigma_z],z\in Z(G)-\{1\}$ and $[J_G]$. Hence it has a pure Hodge structure. 
\end{cor}
\begin{proof}
By assumption on the triviality of $\pi_0(\cZ(L_S))$, we can apply the same argument as for Theorem \ref{prop: RHJ_G_ad}, and get 
\begin{align*}
H_*(\cC_G, \bQ)\cong H_*(\bA^{\Pi}-\{0\}, \bQ)\cong H_*(S^{2n-1}, \bQ).
\end{align*}
Now $J_G$ is from attaching $|\cZ(G)|$ many $2n$-dimensional handles $\chi^{-1}([0])^c_z, z\in \cZ(G)$ to $J_G-\bigcup_{z\in \cZ(G)}\Sigma_z$. Again using a generic cotangent fiber $F_{h'}$ in $|b_G|^{-1}(U_\emptyset)\cong T^*T$, and using the action of $\cZ(G)$ on $J_G$, we get the intersection number in $J_G-\bigcup_{z\in \cZ(G)}\Sigma_z$: 
\begin{align*}
[\partial \chi^{-1}([0])^c_z]\cap [\bR_+\cdot F_{h'}]=\frac{|W|}{|\cZ(G)|},\ \forall z\in \cZ(G). 
\end{align*}
Fix an ordering of $\cZ(G)$ as $\{I=z_1,\cdots, z_{|\cZ(G)|}\}$. Then $H_{2n-1}(J_G, \bQ)=0$, and $H_{2n}(J_G, \bQ)$ has a basis given by the classes of the cycles 
\begin{align}\label{eq: H_2nbasis}
\mathsf{C}_j:=\chi^{-1}([0])_{z_j}-\chi^{-1}([0])_{z_1}-\eta_j,\ 1<j\leq |\cZ(G)|, 
\end{align}
for a $2n$-chain $\eta_j\in C_{2n}(J_G-\Sigma_I,\bQ)$ (clearly, $\eta_j$ is unique up to homologous relations). 
Note that the Borel-Moore cycles $[\Sigma_{z_j}], j\neq 1$, give exactly the dual bases in $H^{2n}(J_{G}, \bQ)$. The proof is complete. 
\end{proof}

Recall that over $\bQ$, the irreducible representations of $\bZ/p\bZ$, for a prime $p$, are just the trivial representation and $\bQ[\bZ/p\bZ]_0$ (the space of $\bQ$-valued functions on $\bZ/p\bZ$ whose values sum up to $0$). As an immediate corollary, we get 
\begin{cor}\label{cor: SLp}
For any prime $p$, we have 
\begin{align*}
H_{-*}(J_{SL_p(\bC)},\bQ)\cong \big(\bQ[\bZ/p\bZ]_0\big)[2(p-1)]\oplus \bQ.
\end{align*}
as representations of $\bZ/p\bZ$. 
\end{cor}

\begin{remark}
When the triviality condition of $\pi_0(\cZ(L_S)), S\subsetneq \Pi$ in Corollary \ref{cor: pi0ZLStrivial,H_*} fails, we do \emph{not} necessarily have $H_{2n}(J_G, \bQ)\cong \bQ^{(|\cZ(G)|-1)}$ . 
\end{remark}


\begin{thebibliography}{99}

\bibitem[AtHi]{AH} M.F. Atiyah, N. Hitchin, 
``The geometry and dynamics of magnetic monopoles." 
M.B. Porter Lectures, Princeton University Press, NJ, 1988.


\bibitem[Bal]{Balibanu} A. Balibanu, 
``The partial compactification of the universal centralizer." 
Selecta Math. (N.S.) 29 (2023), no. 5, Paper No. 85, 36 pp.

\bibitem[BFM]{BFM} R. Bezrukavnikov, M. Finkelberg, I. Mirkovic, 
``Equivariant ($K$-)homology of affine Grassmannian and Toda lattice.'' 
Compositio Math. 141 (2005), no. 3, 746--768.

\bibitem[Bie]{Bielawski} R. Bielawski, `
`HyperK\"ahler structures and group actions." 
J. London Math. Soc. (2) 55 (1997), no. 2, 400--414.

\bibitem[BFN]{BFN} A. Braverman, M. Finkelberg, H. Nakajima, 
``Towards a mathematical definition
of Coulomb branches of 3-dimensional $N = 4$ gauge theories, II." 
Adv. Theor. Math. Phys. 22 (2018), no. 5, 1071--1147.

\bibitem[BZG]{BZG} D. Ben-Zvi, S. Gunningham, 
``Symmetries of categorical representations and the quantum Ng\^o action." 
{\scriptsize\url{https://arxiv.org/abs/1712.01963}}.


\bibitem[Don]{Donaldson} S. K. Donaldson, 
``Nahm's equations and the classification of monopoles." 
Comm. Math. Phys. 96 (1984), no.3, 387--407.


\bibitem[Eti]{Etingof} P. Etingof, 
``Whittaker functions on quantum groups and $q$-deformed Toda operators." 
Amer. Math. Soc. Transl. Ser. 2, 194;
Adv. Math. Sci., 44, Amer. Math. Soc., Providence, RI, 1999, 9--25.

\bibitem[EvJo]{EJ} S. Evens, B. F. Jones, 
``On the wonderful compactification." 
{\scriptsize\url{https://arxiv.org/abs/0801.0456}}.

\bibitem[Fel]{Felisetti} C. Felisetti, 
``Intersection cohomology of the moduli space of Higgs bundles on a genus $2$ curve." 
J. Inst. Math. Jussieu 22 (2023), no. 3, 1037--1086.

\bibitem[Gan]{Gannon} T. Gannon, 
``Classification of nondegenerate G-categories." 
{\scriptsize\url{https://arxiv.org/abs/2206.11247}}.

\bibitem[Gin]{Ginzburg} V. Ginzburg, 
``Nil-Hecke algebras and Whittaker $D$-modules."
Progr. Math. 326, 
Birkh\"{a}user/Springer, Cham, 2018, 137--184.



\bibitem[Hau]{Hausel} T. Hausel, 
``S-duality in hyperk\"{a}hler Hodge theory." 
Oxford University Press, Oxford, 2010, 324--345.



\bibitem[Jin1]{Jin} X. Jin, 
``Homological mirror symmetry for the universal centralizers." 
{\scriptsize\url{https://arxiv.org/abs/2206.09035v3}}.

\bibitem[Jin2]{Jin2} X. Jin, 
``Cohomology of the universal centralizer II: The general case." 
\emph{in preparation}, 2024.

\bibitem[JiYu]{JY} X. Jin, Z. Yun, 
``Mirror symmetry of the affine Toda systems." 
\emph{in preparation}, 2024.

\bibitem[Lon]{Lonergan} G. Lonergan,  
``A Fourier transform for the quantum Toda lattice." 
Selecta Math. (N.S.) 24 (2018), no. 5, 4577--4615.

\bibitem[May]{Maxence} M. Mayrand, 
``Stratified hyperk\"{a}hler spaces and Nahm's equations." 
PhD thesis, Oxford University, 2019.

\bibitem[Ngo]{Ngo} B. C. Ng\^o, 
``Le lemme fondamental pour les alg\`ebres de Lie''. 
Publ. Math. Inst. Hautes \'Etudes Sci. (2010), no. 111, 1--169.


\bibitem[Kos]{Kostant} B. Kostant, 
``On Whittaker vectors and representation theory." 
Invent. Math. 48 (1978), no. 2, 101--184.


 \bibitem[Lus]{Lusztig} G. Lusztig, 
``Coxeter orbits and eigenspaces of Frobenius." 
Invent. Math. 38 (1976/77), no. 2, 101--159.
   
\bibitem[Tel]{Teleman} C. Teleman, 
``Gauge theory and mirror symmetry." 
Proceedings of the International
Congress of Mathematicians--Seoul 2014. Vol. II, 1309--1332.






\end{thebibliography}
\end{document}